\documentclass[a4paper]{amsart}
\usepackage[english]{babel}
\usepackage[utf8]{inputenc}
\usepackage{lmodern}
\usepackage[T1]{fontenc}

\usepackage{lipsum}
\usepackage{amssymb, amsthm, amsmath}
\usepackage{enumitem}
\usepackage{bbm,dsfont}
\usepackage[all]{xy}

\newcommand*\Laplace{\mathop{}\!\mathbin\bigtriangleup}

\newcommand{\K}{\mathbb{K}}
\newcommand{\R}{\mathbb{R}}
\newcommand{\C}{\mathbb{C}}
\newcommand{\N}{\mathbb{N}}
\newcommand{\A}{\mathcal{A}}
\newcommand{\U}{\mathcal{U}}
\newcommand\dual[2]{{\big\langle #1}{,#2\big\rangle }}

\renewcommand{\S}{\mathcal{S}}

\renewcommand{\L}{\mathcal{L}}

\newcommand{\B}{\mathcal{B}}
\renewcommand{\d}{\mathrm{d}}
\renewcommand{\Re}{\operatorname{Re}}

\newcommand{\fra}{\mathfrak{a}}

\renewcommand{\mid}{\, \vert \,}
\renewcommand{\MR}{\textit{MR}\,}

\DeclareMathOperator{\ran}{Ran}
\newcommand\ra{\rightarrow}
\newcommand\lra{\longrightarrow}
\newcommand\restr[2]{{#1}{_{|#2}}}
\theoremstyle{plain}
\newtheorem{theorem}{Theorem}[section]
\newtheorem{proposition}[theorem]{Proposition}
\newtheorem{lemma}[theorem]{Lemma}

\theoremstyle{definition}
\newtheorem{definition}[theorem]{Definition}

\newtheorem{remark}[theorem]{Remark}

\newtheorem{notation}[theorem]{Notation}
\begin{document}
	\title{Regularity properties for evolution family governed by non-autonomous forms}
	\author{
		 Hafida Laasri, Hagen
	}
	
	\address{Lehrgebiete Analysis, Fakult\"at Mathematik und Informatik, Fernuniversit\"at in Hagen, D-58084 Hagen, Germany}
	\email{hafida.laasri@feruni-hagen.de}
	
	\maketitle

\begin{abstract}
This paper gives further regularity properties of the evolution family associated with a non-autonomous evolution equation 
\begin{equation*}\label{Abstract equation}
\dot u(t)+\A(t)u(t)=f(t),\ \ t\in[0,T],\ \
u(0)=u_0,
\end{equation*}
where $\A(t),\ t\in [0,T],$  arise from non-autonomous sesquilinear forms $\fra(t,\cdot,\cdot)$ on a Hilbert space $H$ with constant domain $V\subset H.$ Results on norm continuity, compactness and results on the \textit{Gibbs} character of the evolution family are established. The  abstract results are applied to the Laplacian operator with time dependent Robin boundary conditions.
\end{abstract}



\section*{Introduction\label{s1}}
 \noindent Let $H,V$ be two Hilbert spaces such that 
 $V$ is continuously and densely embedded into 
 $H$ and consider  a  non-autonomous  sesquilinear form  $\fra:[0,T]\times V \times V \to \K$ with 
  
\begin{equation*}\label{eq:continuity-nonaut}
| \fra(t,u,v) | \le M \|u\|_V \|v\|_V \text{ and } \ \Re~ \fra (t,u,u) \ge \alpha \|u\|^2_V
\end{equation*}
for all $t\in[0,T], v\in V$ and for some $\alpha, M>0.$ For each $t\in[0,T]$ we associate with $\fra(t,\cdot,\cdot)$ a unique operator $\A(t)\in \L(V,V')$  such that
\[\fra(t,u,v)=\langle\A(t)u,v\rangle \quad\hbox{ for all } u,v\in V.\]
\par\noindent The non-autonomous Cauchy problem
\begin{equation}\label{Abstract Cauchy problem 0}
\dot{u} (t)+\A(t)u(t)=f(t), \quad u(0)=u_0
\end{equation}
is said to have \textit{$L^2$-maximal regularity in $H$} if for every $f\in L^2(0,T,H)$ and $u_0\in V$ there exists a unique function $u$ belonging  to $L^2(0,T;V)\cap H^1(0,T;H)$ such that $u$ satisfies  (\ref{Abstract Cauchy problem 0}). Maximal regularity is a very important aspect in ongoing research in theory of non-autonomous evolution equations and addressed by many authors. This is due to its applicability in the prove of existence and regularity of solutions. More interestingly, maximal regularity can be used to solve nonlinear evolution equations by means of fixed point theorems. Considering (\ref{Abstract Cauchy problem 0}) on $V',$ in 1961 J. L. Lions proved the following well known result regarding \textit{$L^2$-maximal regularity in $V'$}:

\begin{theorem}[Lions]\label{wellposedness in V'2}
	Given $f\in L^2(0,T;V^\prime)$ and $u_0\in H,$ the problem (\ref{Abstract Cauchy problem 0}) has a unique solution $u \in \MR(V,V'):=L^2(0,T;V)\cap H^1(0,T;V').$ 
\end{theorem}
Theorem \ref{wellposedness in V'2} has been proved by Lions in \cite{Lio61} using a  representation theorem
of linear functionals, known in the literature as \textit{Lions's representation Theorem}. Other proofs of Theorem \ref{wellposedness in V'2} can be found in \cite[ XVIII
Chapter 3, p. 620]{DL88} 
where a Galerkin's method is used or in \cite[Section 5.5]{Tan79} where  a fundamental solution has been constructed. Furthermore, the author gave an alternative proof in \cite[Section 2]{LASA14} using the approach of the integral product of semigroups
developed in \cite{ELKELA11, ELLA13, ELLA15}.
\par\noindent Lions' theorem requires  only  measurability of  $t\mapsto \fra(t,u,v)$  for all $u,v\in V$. However, in applications to boundary problems maximal regularity in $V'$ is not sufficient because it is only the part $A(t)$ of $\A(t)$ in $H$ that realizes the boundary conditions in question. Precisely one is more interested on \textit{$L^2$-maximal regularity in $H,$} i.e., the solution $u$ of (\ref{Abstract Cauchy problem 0}) belong to $H^1(0,T;H)$ if $f\in L^2(0,T; H)$ and $u_0\in V.$ The problem of $L^2$-maximal regularity in $H$ was initiated  by Lions in \cite[p.\ 68]{Lio61} for $u_0=0$ and $\fra$ is symmetric. In general, we have to impose more regularity on the form $\fra$ then measurability of the form is not sufficient \cite{D1,ADF17}. However, under additional regularity assumptions on the form $\fra,$   the
initial value $u_0$ and the inhomogeneity $f,$ some positive results were already done by Lions in
\cite[p.~68, p.~94, ]{Lio61}, \cite[Theorem~1.1, p.~129]{Lio61} and
\cite[Theorem~5.1, p.~138]{Lio61} and by Bardos \cite{Bar71}. More recently, this
problem has been studied  with some progress and different
approaches \cite{ADLO14, Ar-Mo15, O15, D2, OH14, OS10, ELLA15, Di-Za16, Au-Eg16, Fa17}. Results on multiplicative perturbation are established in \cite{ADLO14, D2, AuJaLa14}. See also the recent review paper \cite{ADF17} for more details and references.
 \medskip
 
 In this paper we are mainly interested by further regularity of the \textit{evolution family} generated by $\A(t), t\in [0,T].$  Recall that it is well known that, under suitable conditions, the solution of a non-autonomous linear evolution equation  may  be given by a  strongly continuous evolution family \[\Big\{U(t,s):\  0\leq s\leq t\leq T\Big\}\subset \L(H)\] i.e., a family that has the following properties:
 \begin{itemize}
 	\item[$(i)$] $U(t,t)=I$ and $U(t,s)= U(t,r)U(r,s)$ for every $0\leq r\leq s\leq t\leq T,$ 
 	\item [$(ii)$] for every $x\in H$ the function $(t,s)\mapsto U(t,s)x$ is continuous into $H$ for $0\leq s\leq t\leq T.$
 \end{itemize}
 In the autonomous case, i.e., if $\fra(t,\cdot,\cdot)=\fra_0(\cdot,\cdot)$ for all $t\in[0,T],$ then one knows that $-A_0,$ the operator associated with $\fra_0$ in $H,$ generates a holomorphic $C_0$-semigroup  $(T(t))_{t\geq 0}$ in $H.$ In this case $U(t,s):=T(t-s)$ yields a strongly continuous evolution family on $H.$  
 
 \par \noindent The study of regularity properties of the evolution family with respect to $(t.s)$ in general Banach spaces has been investigated in the Literature in the case of constant domains by Komatsu \cite{H.Ko61}  and Lunardi \cite{Lu89} and by Acquistapace \cite{Ac88} for the case of time-dependent domains. 
 \medskip
 \par\noindent  Consider a non-autonomous sesquilinear form  $\fra(t,.,.): V\times V\lra \C$ form such that there exists  $0\leq \gamma<1$ and a continuous function $\omega:[0,T]\longrightarrow [0,+\infty)$ with
	\begin{equation*}\label{eq: thm Arendt-Monniaux}\sup_{t\in[0,T]} \frac{\omega(t)}{t^{\gamma/2}}<\infty \quad \text{ and }
	\int_0^T\frac{\omega(t)}{t^{1+\gamma/2}}<\infty
	\end{equation*}
	such that
	\begin{equation*}
	|\fra(t,u,v)-\fra(s,u,v)| \le\omega(|t-s|) \Vert u\Vert_{V} \Vert v\Vert_{V_\gamma}\quad \ (t,s\in[0,T], u,v\in V),
	\end{equation*}
	where $V_\gamma:=[H,V]$ is the complex interpolation space.
	 Under this assumptions, (\ref{Abstract Cauchy problem 0}) has $L^2$-maximal regularity in $H$ provided $\fra(0; \cdot,\cdot)$ has the square root property by a recent results of Arendt and Monniaux \cite{Ar-Mo15}. In this paper we push their analysis forward by discussing additional regularity properties of the solution. 

\noindent 
In Section \ref{Sec2 Norm continuity} we establishes that in this case the solution of the non-autonomous Cauchy problem (\ref{Abstract Cauchy problem 0}) is governed by a strongly continuous evolution family $\U(\cdot,\cdot)$ in $H$ such that    
the mapping $(t,s)\mapsto \U(t,s)$ is norm continuous with value in $\L(V)$ for $0\leq s<t\leq T.$ If in addition the embedding $V\subset H$ is compact then we show in Section \ref{Sec3 Gipps property} that the operator $\U(t,s)$ is compact with value in $\L(V),$ respectively in  $\L(H),$ for each $0\leq s<t\leq T.$ This implies, in particular, that  $(t,s)\mapsto \U(t,s)$ is also norm continuous with value in $\L(H)$ for $0\leq s<t\leq T$ provided that $V\subset H$ is compact. An other important result of Section \ref{Sec3 Gipps property} concerns the Gibbs character of the evolution family $\U(\cdot,\cdot).$ Recall that the trace class operators is a special case of $p$-Schatten class operators $\S_p(H)$ when $p=1.$ Then we call that an evolution family is \textit{Gibbs} if   $\U(t,s)$ is a trace operator for each $0\leq s<t\leq T$. In the autonomous situation, the concept of Gibbs semigroups has been introduced by Dietrich A. Uhlenbrock \cite{Uh71}, and it is known that this class play an important role in spectral theory and mathematics physics. In Section \ref{Sec3 Gipps property} we show that the evolution family associated with (\ref{Abstract Cauchy problem 0}) is a Gibbs evolution family provided that there exists some $p\in (1,\infty)$ such that the embedding $V\subset H$ is a $p$-Schatten operator.  We apply our abstract results in Section \ref{S application} to parabolic equations governed by the Laplacian operator with time dependent Robin boundary conditions. 
\section{Preliminary results \label{Approximation}}

Throughout this paper $H,V$ are two Hilbert spaces over $\mathbb C$ such that  $V \underset d \hookrightarrow H;$ i.e., $V$ is  densely embedded into $H$ and
\begin{equation*}\label{eq:V_dense_in_H}
\|u\| \le c_H \|u\| _V \quad (u \in V)
\end{equation*}
for some constant $c_H>0.$ Let $V'$ denote  the antidual of $V.$  The duality
between $V'$ and $V$ is denoted by $\langle ., . \rangle$. As usual, by identifying $H$ with  $H',$ we have $V\hookrightarrow H\cong H'\hookrightarrow V'.$  These embeddings are continuous and
\begin{equation*}\label{eq:H_dense_in_V'}
\|f\|_{V'} \le c_H \|f\|  \quad (f \in V')
\end{equation*}
see e.g., \cite{Bre11}. We denote by $(\cdot \mid \cdot)_V$ the scalar product and $\|\cdot\|_V$ the norm
on $V$ and by $(\cdot \mid \cdot), \|\cdot\|$ the corresponding quantities in $H.$  Let $
\fra: [0,T]\times V\times V \to \C$ be a continuous and coercive non-autonomous sesquilinear form, i.e., $\fra(.,u,v)$ is measurable,
\begin{equation}\label{eq:continuity-nonaut}
|\fra(t,u,v)| \le M \Vert u\Vert_V \Vert v\Vert_V
\end{equation}
and
\begin{equation}\label{eq:Ellipticity-nonaut}
\Re ~\fra (t,u,u)\ge \alpha \|u\|^2_V 
\end{equation}
for some constants  $\alpha, M> 0$ and for all $t\in [0,T], u,v\in V.$
By the Lax-Milgram theorem, for each $t\in[0,T]$ there exists an isomorphism $\A(t):V\to V^\prime$ such that
$\langle \A(t) u, v \rangle = \fra(t,u,v)$ for all $u,v \in V.$ It is well known that $-\A(t),$
regarding as unbounded operator with domain $V,$  generates a bounded holomorphic semigroup $e^{-\cdot\A(t)}$  of angle $\theta:=\frac{\pi}{2}-\arctan(\frac{M}{\alpha})$ on $V'$. We call $\A(t)$ the operator associated with $\fra(t,\cdot,\cdot)$ on $V^\prime.$ We have also to consider the operator $A(t)$ associated with $\fra(t,\cdot,\cdot)$ on  $H:$
\begin{align*}
D(A(t)) := {}& \{ u\in V : \A(t) u \in H \}\\
A(t) u = {}& \A(t) u.
\end{align*}
i.e., $A(t)$ is the part of $\A(t)$ in $H.$ Then  $-A(t)$ generates a holomorphic $C_0$-semigroup (of angle $\theta$) $e^{-\cdot A(t)}$ on $H$ which is the restriction to $H$ of $e^{-\cdot A(t)},$ and we have
\begin{equation}\label{analytic representation}
e^{-\cdot A(t)}=\frac{1}{2i\pi}\int_\Gamma e^{\cdot\mu} (\mu+A(t))^{-1}{\rm d}\mu
\end{equation}
where $\Gamma:=\{re^{\pm \varphi}:\ r>0\}$ for some fixed $\varphi\in (\theta,\frac{\pi}{2})$
(see e.g. \cite[Lecture 7]{Ar06},\cite{Ka},\cite[Chapter 1]{Ou05} or \cite[Chap.\ 2]{Tan79}).
\par 
We assume in addition, that there exists $0\leq \gamma< 1$ and a continuous function   $\omega:[0,T]\longrightarrow [0,+\infty)$ with
\begin{equation}\label{eq 1:Dini-condition}  \sup_{t\in[0,T]} \frac{\omega(t)}{t^{\gamma/2}}<\infty, \end{equation}
\begin{equation}\label{eq 2:Dini-condition}
\int_0^T\frac{\omega(t)}{t^{1+\gamma/2}} {\rm d}t<\infty
\end{equation}
and
\begin{equation}\label{eq 3:Dini-condition}
|\fra(t,u,v)-\fra(s,u,v)| \le\omega(|t-s|) \Vert u\Vert_{V} \Vert v\Vert_{V_\gamma}
\end{equation}
for all $t,s\in[0,T]$ and for all $u,v\in V$ where $V_\gamma:=[H,V]_\gamma$ is the complex interpolation space.
Note that

\[V\hookrightarrow V_\gamma \hookrightarrow H\hookrightarrow V_\gamma'\hookrightarrow V'\] with continuous  embeddings.  
\begin{remark}\label{key rematk}
	Remark that conditions (\ref{eq 1:Dini-condition}) and (\ref{eq 2:Dini-condition}) implies that 
	\begin{equation}\label{Key rmk1:Dini-condition}
	\int_0^T\frac{\omega(t)^2}{t^{1+\gamma}} {\rm d}t<\infty
	\end{equation}
	and for each $\varepsilon>0$ there exists $\delta_0>0$ such that 
	\begin{equation}\label{Key rmk2:Dini-condition}
	\int_0^{\delta_0}\frac{\omega(t)}{t^{1+\gamma/2}} {\rm d}t<\varepsilon.
	\end{equation}
\end{remark}

\par The following proposition is of great interest for this paper.

\begin{proposition}\label{lemma: estimations for general from}\cite[Section 2]{Ar-Mo15} Let $b$ be any sesquilinear form that satisfies assumptions (\ref{eq:continuity-nonaut})-(\ref{eq:Ellipticity-nonaut}) with the same constants $M$ and $\alpha$ and let $\gamma\in [0,1[.$ Let $\B$ and $B$ be the associated operators  on $V'$ and $H,$ respectively. Then the following estimates holds 

	\begin{enumerate}
		\item \label{Eq1: estimation resolvent}$\displaystyle \|(\lambda-\B)^{-1}\|_{\L(V_{\gamma}',H)}\leq \frac{c}{(1+\mid\lambda\mid)^{1-\frac{\gamma}{2}}},$
		\item \label{Eq2: estimation resolvent}$\displaystyle \|(\lambda-\B)^{-1}\|_{\L(V)}\leq \frac{c}{1+\mid\lambda\mid},$
		\item \label{Eq3: estimation resolvent}$\displaystyle \|(\lambda-\B)^{-1}\|_{\L(H,V)}\leq \frac{c}{(1+\mid\lambda\mid)^{\frac{1}{2}}},$
		\item \label{Eq4: estimation resolvent} $\displaystyle \|(\lambda-\B)^{-1}\|_{\L(V',H)}\leq \frac{c}{(1+\mid\lambda\mid)^{\frac{1}{2}}},$
	\item \label{Eq5: estimation resolvent} $\displaystyle \|(\lambda-\B)^{-1}\|_{\L(V',V)}\leq c,$
		\item\label{Eq6: estimation resolvent} $\displaystyle \|(\lambda-\B)^{-1}\|_{\L(V'_\gamma,V)}\leq \frac{c}{(1+\mid\lambda\mid)^{\frac{1-\gamma}{2}}},$
		\item \label{analytic estimation gamma=1} $\|e^{-s\B}\|_{\L(V',V)}\leq \displaystyle\frac{c}{s}$
		\item \label{Eq2: estimation  semigroup}
		$\displaystyle\|e^{-s\B}\|_{\L(V_\gamma',H)}\leq\frac{c}{s^{\gamma/2}},$
		\item \label{Eq3: estimation  semigroup}
		$\displaystyle\|e^{-sB}\|_{\L(V_\gamma',V)}\leq\frac{c}{s^{\frac{1+\gamma}{2}}},$
		\item \label{analytic estimation}
		$\displaystyle\|Be^{-sB}\|_{\L(H)}\leq \frac{c }{s},$
		\item \label{analytic estimation in V}
		$\displaystyle\|e^{-sB}\|_{\L(V)}\leq c$
	\end{enumerate}
	for each $s\geq 0, \lambda\notin \Sigma_\theta:=\{re^{i\varphi}: r>0, |\varphi|<\theta\}$ and for some constant $c>0$ which depend only  on $M,\alpha, \gamma$ and $c_H.$ 
	
\end{proposition}
\begin{remark}\label{Remark: estimations for general from} All estimates in Proposition \ref{lemma: estimations for general from} holds for $\A(t)$ with constant independent of  $t\in[0,T],$ since $\fra$ satisfies (\ref{eq:continuity-nonaut})-(\ref{eq:Ellipticity-nonaut}) with the same constants $M$ and $\alpha,$ also $\gamma$ and $c_H$  does not depend on $t\in[0,T].$
\end{remark}
 If $b:V\times V\to \C$ is a coercive and bounded form and  $\B\in\L(V,V')$ the associated operator on $V',$ then one defines  $\B^{-1/2}\in \L(V')$ by 
 \begin{equation}\label{formula for square root} 
 \B^{-1/2}u:=\frac{1}{\pi}\int_0^\infty t^{-1/2}(t-\B)^{-1} {\d }t,\quad \text{ for all } u\in V',\end{equation}
 see \cite[(6.4), Section 2.6]{Pa} or \cite[(3.52), Section 3.8]{ABHN11}. The operator $\B^{-1/2}$ is also  given by the following formula
  \begin{equation}\label{formula for square root2}
   \B^{-1/2}u=\frac{1}{\sqrt{\pi}}\int_0^\infty t^{-\frac{1}{2}}e^{-t\B}u \d t,\quad \text{ for all } u\in V'\end{equation}
 see \cite[(6.9), Section 2.6]{Pa}. Moreover, $\B^{-1/2}$ is one-to-one and one defines $\B^{1/2}$ by 
 
\[ D(\B^{1/2}):=\ran (\B^{-1/2}) \ \text{ and } \ \B^{1/2}u=(\B^{-1/2})^{-1}u \quad \text{ for all } u\in  D(\B^{1/2}).\]
Let $B$ be the part of $\B$ in $H.$ Then $B^{-1/2}u=\B^{-1/2}u$ for each $u\in H$ and $B^{1/2}=(B^{-1/2})^{-1}$ with domain $D(B^{1/2})=\ran(B^{-1/2}).$ We recall that $b:V\times V\rightarrow \C$ has the \textit{square root property} if the following equivalent conditions are satisfied
	\begin{description}
		\item [\ $(i)$] $D(B^{1/2})=V.$
		\item [\ $(ii)$] $D(\B^{1/2})=H.$
\end{description}
 Conditions $(i)$ and $(ii)$ above are indeed equivalent since $\B^{-1/2}\B^{-1/2}=\B^{-1}$ is an isomorphism from $V'$ into $V.$ If a coercive and bounded sesquilinear form $b:V\times V\rightarrow \C$ has the square root property, then $\B^{-1/2}$ is a isomorphism from $V'$ into $H$ with inverse $\B^{1/2},$ and $B^{-1/2}$ is a isomorphism from $H$ into $V$ with inverse $B^{1/2}.$ Further, $B^{1/2}$ is  the part of $\B^{1/2}$ in $H.$

\begin{lemma}\label{lemma: Holder continuity for square} Assume that  $\fra:[0,T]\times V\times V\rightarrow \C$ satisfies  (\ref{eq:continuity-nonaut})-(\ref{eq 3:Dini-condition})  and  $D(A^{1/2}(t))=V$ for every $t\in [0,T].$ Then there exists a constant $c_0>0$  such that 
	 \begin{equation}\label{eq1 lemma: Holder continuity for square}\|\A^{-1/2}(t)-\A^{-1/2}(s)\|_{\L(V',H)}\leq c_0 w(|t-s|),\end{equation}
 
	  \begin{equation}\label{eq12 lemma: Holder continuity for square}\|A^{-1/2}(t)-A^{-1/2}(s)\|_{\L(H,V)}\leq c_0 w(|t-s|)\end{equation}
	and 
	\begin{equation}\label{eq2 lemma: Holder continuity for square} \|A^{1/2}(t)-A^{1/2}(s)\|_{\L(V,H)}\leq  c_0 w(|t-s|)\end{equation}
for all $t,s\in[0,T].$
\end{lemma}
\begin{proof}
	Let $u\in V'$ and $t,s\in[0,T].$ 
Then using (\ref{formula for square root2}) we obtain
	\begin{align*}
	 \A^{-1/2}(t)u-\A^{-1/2}(s)u&=\frac{1}{\sqrt{\pi}}\int_0^\infty \frac{1}{\sqrt\tau}\big(e^{-\tau \A(t)}u-e^{-\tau \A(s)}u\big)\d \tau
	 \\&=\frac{1}{\sqrt{\pi}}\int_0^\infty \frac{1}{2\pi i\sqrt\tau}\int_{\Gamma}e^{-\tau \lambda}(\lambda-\A(t))^{-1}(\A(t)-\A(s))(\lambda-\A(t))^{-1}u \d\lambda \d\tau
\end{align*}
By Proposition \ref{lemma: estimations for general from}-(\ref{Eq1: estimation resolvent}), (\ref{Eq5: estimation resolvent}) and assumption (\ref{eq 3:Dini-condition}) 
\begin{align*}
\|&(\lambda-\A(t))^{-1}(\A(t)-\A(s))(\lambda-\A(t))^{-1}u\|_H
\\&\leq \|(\lambda-\A(t))^{-1}\|_{_{\L(V_\gamma',H)}}\|\A(t)-\A(s)\|_{\L(V,V_\gamma')}\|(\lambda-\A(t))^{-1}\|_{\L(V',V)}\|u\|_{V'}
\\&\leq c\omega(|t-s|)\frac{1}{(1+|\lambda|)^{1-\frac{\gamma}{2}}}\|u\|_{V'}
\end{align*}

for each $\lambda\in \Gamma$ and some constant $c>0$ depending only on $M,\alpha, \gamma$ and $c_H.$ Then
	 \begin{align*}
	 \|\A^{-1/2}(t)u-\A^{-1/2}(s)u\|_{H}
	 &\leq c\omega(|t-s|)\|u\|_{V'} \int_0^\infty \int_{0}^\infty \frac{1}{\sqrt\tau} e^{-\tau r\cos{\theta}}\frac{1}{(1+r)^{1-\frac{\gamma}{2}}}\d\tau\d r
	 \\&=c \omega(|t-s|)\|u\|_{V'}\int_0^\infty \int_{0}^\infty \frac{\sqrt r}{\sqrt\tau} e^{-s \cos{\theta}}\frac{\d s}{r}\frac{1}{(1+r)^{1-\frac{\gamma}{2}}}\d r 
	 \\&\leq  c\omega(|s-t|)\|u\|_{V'}\int_0^\infty\frac{1}{\sqrt r}\frac{1}{(1+r)^{1-\frac{\gamma}{2}}}\d r 
	 \\&\leq c\omega(|s-t|)\|u\|_{V'}
	 \end{align*}
	 for some constant $c>0$ depending only on $M,\alpha, \gamma$ and $c_H.$ This proves  inequality  (\ref{eq1 lemma: Holder continuity for square}), the second inequality also follows easily by a similar way. Let now prove the last one. First remark that (\ref{eq 3:Dini-condition})  implies, in particular, that
	 \[\|\A(t)-\A(s)\|_{\L(V,V')}\leq c\omega(|t-s|).\]
On the other hand,  $\A^{-1/2}(\cdot):[0,T]\rightarrow\L(V',H)$ is   continuously (\ref{eq1 lemma: Holder continuity for square}) since $\omega(|t-s|)\to 0$ as $t\to s$ by assumption. Thus 
	 \begin{equation}\label{constant kappa}\kappa:=\sup_{r\in[0,T]}\|\A^{-1/2}(r)\|_{\L(V',H)}< \infty.\end{equation}
Since $A^{1/2}(\cdot)u=\A^{-1/2}(\cdot)\A(\cdot)u$ we deduce from (\ref{eq:continuity-nonaut}) and  (\ref{eq1 lemma: Holder continuity for square}) that  
	 \begin{align*}
	 \|A^{1/2}(t)u&-A^{1/2}(s)u\|_{H}=\|\A^{-1/2}(t)\A(t)u-\A^{-1/2}(s)\A(s)u\|_{H}
	 \\&\leq \|\A^{-1/2}(t)\|_{\L(V',H)}\|(\A(t)u-\A(s)u)\|_{V'}+\|(\A^{-1/2}(t)-\A^{-1/2}(s))\|_{\L(V',H)}\|\A(s)u\|_{V'}
	 \\&\leq c_0\omega(|t-s|)\|u\|_{V}
	 \end{align*}
	 where the constant $c_0>0$ depends only on $M,\alpha, \gamma, c_H$ and $\kappa,$ ans thus the claim is proved.
	 \end{proof}
	 The following lemma is a direct consequence of Lemma \ref{lemma: Holder continuity for square}.
\begin{lemma}
\label{lemma: equivalent norms on V} Assume that  $\fra:[0,T]\times V\times V\rightarrow \C$ satisfies  (\ref{eq:continuity-nonaut})-(\ref{eq 3:Dini-condition})  and  $D(A^{1/2}(t))=V$ for every $t\in [0,T].$ Then 
there exists a constant $\sigma>0$ such that
\[\frac{1}{\sigma}\|u\|_V\leq \|A^{1/2}(t)u\|_H\leq\sigma \|u\|_V \]
for all $u\in V$ and every $t\in [0,T].$
\end{lemma}
\begin{notation} To keep notations simple as possible  we will in the sequel denote all positive constants depending on $M,\alpha, \gamma, c_H,T$ and $\sigma$ that appear in proofs and theorems uniformly as $c>0.$
\end{notation}
\section{Norm continuous evolution family}\label{Sec2 Norm continuity}

  Let $
 \fra: [0,T]\times V\times V \to \C$ be a bounded and coercive non-autonomous sesquilinear form. Then we known from Theorem \ref{wellposedness in V'2}
 that for each given $ s\in[0,T)$ and $x\in H,$ the Cauchy problem  
 \begin{equation}\label{evolution equation u(s)=x}
 \dot{u}(t)+\A(t)u(t)=0\ \ \hbox{a.e. on}\ [s,T],\ \   u(s)=x,
 \end{equation}
 has a unique solution $u\in\MR(V,V'):=\MR (s,T;V,V'):=L^2(s,T;V)\cap H^1(s,T;V').$ Recall that  the maximal regularity space $\MR(V,V')$ is continuously embedded in $C([s,T],H).$
Therefore, for every $(t,s)\in\overline{\Delta}$ and every $x\in H$ we can define  
 
 \begin{equation}\label{evolution family}
 \mathcal U(t,s)x:=u(t),
\end{equation}
  where $u$ is the unique solution in $MR(V,V')$ of (\ref{evolution equation u(s)=x}) and $\Delta:=\{ (t,s)\in[0,T]^2:\  t< s\}.$ 
 \begin{proposition} 
The family  $\{\mathcal U(t,s):\ (t,s)\in\overline{\Delta}\}$ yields a contractive, strongly continuous evolution family on $H.$ Moreover, 
  for each $f\in L^2(0,T,H)$ 
 \[v(t)=\int_0^t  \mathcal U(t,r)f(r)dr\]
 is the unique solution in $MR(V,V')$ of the inhomogeneous problem 
 \begin{equation}\label{inhomegenuous problem}
 \dot{v}(t)+\A(t)v(t)=f(t)\ \ \hbox{a.e. on}\ [0,T],\ \   u(0)=0,
 \end{equation} 
\end{proposition}
\begin{proof}
The last assertion and the fact that $\{\mathcal U(t,s):\ (t,s)\in\overline{\Delta}\}$ yields a bounded, strongly continuous evolution family on $H$ follows immediately  from \cite[Proposition 2.3, Proposition 2.4]{ACFP}. It remains to show that 

 \[\|\mathcal U(t,s) x\|\leq \|x\|, \quad\text{ for all  $x\in H$  and $(t,s)\in\overline{\Delta}.$}\]
Let  $x\in H$ and $s\in[0,T).$ Since  $\U(\cdot,s)x\in \MR (V,V')$ then
the function  $\|\U(\cdot,s)\|^2$ is absolutely continuous on $[s,T]$ and
\begin{equation}\label{eq:chain_rule V'}
\frac{d}{dt}\| \U(t,s)\|^2=2\Re\dual{\dot\U(t,s)}{\U(t,s)},
\end{equation}
see e.g., \cite[Chapter III, Proposition 1.2]{Sho97} or \cite[Lemma 5.5.1]{Tan79}. It follows
 	\begin{align*}
 	\frac{d}{dt}\|\mathcal U(t,s)x\|^2&=2\Re \dual{-\A(t)\mathcal U(t,s)x}{\mathcal U(t,s)x}
 	\\&=-2\Re \fra(t,\mathcal U(t,s)x,\mathcal U(t,s)x)
 	\leq -2\alpha \|\mathcal U(t,s)x\|^2_V\leq 0
 	\end{align*}
 	for almost every $t\in [s,T].$ Integrating this inequality on $]s,t[$ and using (\ref{eq:Ellipticity-nonaut})  we obtain 
 	 
 	\[\|\mathcal U(t,s)x\|^2-\|\mathcal U(s,s)x\|^2\leq -2\alpha\int_s^t   \|\mathcal U(r,s)x\|^2_V\d r\leq 0,\]
 and the claim follows.
 \end{proof}
 
 %

\paragraph\noindent  For the remainder of this section, we assume that  $\fra:[0,T]\times V\times V\rightarrow \C$ satisfies  (\ref{eq:continuity-nonaut})-(\ref{eq 3:Dini-condition})  and $D(A(0)^{1/2})=V$ and we set
\begin{equation}\label{restriction of EVF on V}
U(t,s)x:=\mathcal U(\cdot,\cdot)x\quad \text{ for } x\in V.
\end{equation}
Then it is well known that $U(\cdot,s)x\in MR_{\fra,s}\subset  C(s,T,V)$ 
where 

\[MR_{\fra,s}:=\{u\in H^1(s,T,H)\cap L^2(s,T,V):\ u(s)\in V \text{ and }  \A(\cdot)u(\cdot)\in L^2(s,T,H)\},\]
see \cite[Theorems 4.1, 4.2]{Ar-Mo15}.  The  main result of this section states that $\Big\{ U(t,s): (t,s)\in\overline{\Delta}\Big\}\subset\L(V)$ define  a \textit{norm continuous} evolution family on $V,$ that is  $U(\cdot,\cdot)\in C(\Delta, \L(V)).$ To this end, we recall that the solution $u=U(\cdot,s)$ of (\ref{evolution equation u(s)=x}) satisfies the following key formula
\begin{equation}\label{Formula for the solution}
U(t,s)x=e^{-(t-s)A(t)}x+\int_{s}^te^{-(t-r)A(t)}(\A(t)-\A(r))U(r,s)x{\rm  d}r
\end{equation}
for all $(s,t)\in\Delta$ and $x\in V.$ This formula is due to Acquistapace and Terreni \cite{Ac-Ter87} and was proved in a more general setting in \cite[Proposition 3.5]{Ar-Mo15}. First we introduce the following notations:

\begin{equation*}\label{Formula for the solution 1}
U_{1}(t,s)x:=e^{-(t-s)A(t)}x, \quad 
\end{equation*}
and 
\begin{equation*}\label{Formula for the solution 2}
U_{2}(t,s)x:=\int_{s}^te^{-(t-r)\A(t)}(\A(t)-\A(r))U(r,s)x{\rm  d}r.
\end{equation*}
Moreover, consider the operators $P_s: C(s,T,V)\ra  C(s,T,V), s\in[0,T),$ defined for each $h\in C(s,T,V)$ and $\tau\in[s,T]$ by
\[(P_sh)(\tau):=\int_{s}^\tau e^{-(\tau-r)\A(\tau)}(\A(\tau)-\A(r))h(r){\rm  d}r.\]
Let $c_0>0$ be arbitrary and fixed. Then replacing $\A(t)$ with $\A(t)+\mu$ for $\mu\in\R$ we may assume without loss of generality that 
\begin{equation}\label{constant c0}\|A(t)^{-\frac{1}{2}}\|_{\L(V)}\leq c_0,\quad \text{  for all } t\in[0,T].
\end{equation}
 With this notations we have the following first result. 
 
 \begin{lemma}\label{main theorem: step2}
 The operator $P_s$  is well defined and $\|P_s\|_{\L(C(s,T,V))}<\frac{1}{4}$ by choosing $c_0>0$ in (\ref{constant c0}) to be small enough.
 \end{lemma}
 This Lemma has been proved in \cite[Step.3 of the proof of Theorem 4.4]{Ar-Mo15}. For completeness we will carry out the proof.  
 \begin{proof} (of Lemma \ref{main theorem: step2})	For each $s\in[0,T],$ $P_s$ is well defined  thanks to \cite[Lemma 4.5]{Ar-Mo15} and the Lebesgue domination theorem. Moreover,  since  (\ref{eq 3:Dini-condition}), Proposition \ref{lemma: estimations for general from}-(\ref{Eq3: estimation  semigroup}) imply \begin{align}
 	\|\nonumber&\displaystyle A^{1/2}(t_2)e^{-(t_2-t_1)\A(t_2)}(\A(t_2)-\A(t_1))h(t_1)x\|_V
 	\\\nonumber\displaystyle =&\|A^{1/2}(t_2)\displaystyle e^{-\frac{(t_2-t_1)}{2}\A(t_2)}e^{-\frac{(t_2-t_1)}{2}\A(\tau)}(\A(t_2)-\A(t_1))h(r)x\|_V
 	\\&\label{proof main thm: eq key}\leq c\frac{\omega(t_2-t_1)}{(t_2-t_1)^{1+\frac{\gamma}{2}}})\|x\|_V
 	\end{align} for all $t_1,t_2\in[0,T]$ with $t_1<t_2,$ it follows
 	
 	\[\|P_s h\|_\infty\leq c \|A^{-1/2}(\tau)\|_{\L(V)}\Big(\int_0^T \frac{\omega(\sigma)}{\sigma^{1+\frac{\gamma}{2}}} \d\sigma\Big)\|h\|_\infty\leq cc_0\|h\|_\infty\]
 	The last inequality above holds thanks to (\ref{eq 2:Dini-condition}).  Choosing then $c_0>0$ small enough we obtain that  $\|P_s\|_{\L(C(s,T,V))}<\frac{1}{4}.$
 \end{proof}

 \begin{lemma}\label{main theorem: step1}  Assume $\fra:[0,T]\times V\times V\rightarrow \C$ satisfies  (\ref{eq:continuity-nonaut})-(\ref{eq 3:Dini-condition})  and $D(A(0)^{1/2})=V.$ Then 
 	\begin{equation}\label{Step 1main proof: Eq1}
 	\|U_1(t,s)-U_1(t',s)\|_{\L(V)} \leq c\Big[\log\Big(\frac{t'-s}{t-s}\Big)+\omega(|t-t'|)\Big]
 	\end{equation}
 	holds for each $s\in[0,T]$ and all $0\leq s<t\leq t'\leq T.$ 
 \end{lemma}
 \begin{proof} Let $x\in V$ and $0\leq s<t\leq t'\leq T.$ Because of assumptions (\ref{eq:continuity-nonaut})-(\ref{eq 3:Dini-condition}) we have $D(A(t)^{1/2})=V$ for all $t\in [0,T]$ \cite[Proposition 2.5]{Ar-Mo15}. Then we deduce from (\ref{eq 3:Dini-condition}), Lemma \ref{lemma: equivalent norms on V} and Proposition \ref{lemma: estimations for general from}.(\ref{analytic estimation}) that
 	\begin{align*}
 	\|e^{-(t-s)A(t)}x-e^{-(t'-s)A(t)}x\|_V&\leq\int_{t-s}^{t'-s}\|A(t)e^{-rA(t)}x{\d}r\|_V
 	\\&\leq \sigma^{-1}\int_{t-s}^{t'-s}\|A^{1/2}(t)A(t)e^{-rA(t)}x\|_H{\d}r
 	\\&\leq \sigma^{-1}\int_{t-s}^{t'-s}\|A(t)e^{-rA(t)}A^{1/2}(t)x\|_H{\d}r
 	\\&\leq c\sigma^{-1}\int_{t-s}^{t'-s}\frac{1}{r}\|A^{1/2}(t)x\|_H{\d}r
 	\\&\leq c\sigma^{-1}\sigma \int_{t-s}^{t'-s}\frac{1}{r}{\d}r\|x\|_V
 	\\&=c\log\big(\frac{t'-s}{t-s}\big)\|x\|_V.
 	\end{align*}
 	On the other hand, since 
 	\[e^{-(t'-s)A(t)}x-e^{-(t'-s)A(t')}x=\frac{1}{2\pi i}\int_{\Gamma} e^{-(t'-s)\lambda}(\lambda-\A(t))^{-1}(\A(t)-\A(t'))(\lambda-\A(t'))^{-1}x\d\lambda.\]
 	by (\ref{analytic representation}), we deduce from  assumption (\ref{eq 3:Dini-condition}) and Proposition \ref{lemma: estimations for general from}.(\ref{Eq2: estimation resolvent}),(\ref{Eq6: estimation resolvent})  that 
 	\begin{align*}
 	\|e^{-(t'-s)A(t)}x-e^{-(t'-s)A(t')}x\|_V
 	\leq &c\int_{\Gamma} e^{-(t'-s)\Re\lambda}\frac{\omega(|t-t'|)}{(1+\lambda)^{\frac{3-\gamma}{2}}}\d\lambda\|x\|_V
 	\\&=c\omega(|t-t'|)\int_{0}^\infty \frac{e^{-(t'-s)r\cos{\varphi}}}{(1+\lambda)^{\frac{3-\gamma}{2}}}{\d }r\|x\|_V
 	\\&\leq{c}\omega(|t-t'|)\int_{0}^\infty \frac{1}{(1+r)^{\frac{3-\gamma}{2}}}{\d } r\|x\|_V
 	\end{align*}
 	Now (\ref{Step 1main proof: Eq1}) follows by the triangle inequality.
 \end{proof}


 Now we arrive to our crucial technical Lemma. 
 \begin{lemma}\label{main theorem: step3}
 For all $0\leq s<t\leq t'\leq T,$  $h\in C(s,T,V)$ and $0<\delta_0<t-s$ we have   \begin{equation}\label{main thm : main estimates} \|(P_sh)(t')-(P_sh)(t)\|_V\leq c \Big[\kappa_s(t'-t)+\int_0^{\delta_0}\frac{\omega(\sigma)}{\sigma^{1+\frac{\gamma}{2}}}\d\sigma \Big]\|h\|_\infty,\end{equation}
where $ \kappa_s: [0,T]\to \R_+$ with $\kappa_s(r)\to 0$ as $r\to 0.$ 
\end{lemma}
\begin{proof}
We write $(P_sh)(t')-(P_sh)(t)$  as follows 
\begin{align}
\nonumber &(P_sh)(t')-(P_sh)(t)\\\nonumber&=\int_{0}^{t'-s}e^{-(t'-s-r)\A(t')}(\A(t')-\A(r+s))h(r+s){\rm  d}r-\int_{0}^{t-s}e^{-(t-s-r)\A(t)}(\A(t)-\A(r+s))h(r+s){\rm  d}r
\\&=\label{crucial step: eq1}\int_{0}^{t'-s-\delta_0}e^{-(t'-s-r)\A(t')}(\A(t')-\A(r+s))h(r+s){\rm  d}r-\int_{0}^{t-s-\delta_0}e^{-(t-r)\A(t)}(\A(t)-\A(r+s))h(r+s){\rm  d}r
\\&\label{crucial step: eq2}\quad +\int_{t'-s-\delta_0}^{t'-s}e^{-(t'-s-r)\A(t')}(\A(t')-\A(r+s))h(r+s){\rm  d}r-\int_{t-s-\delta_0}^{t-s}e^{-(t-s-r)\A(t)}(\A(t)-\A(r+s))h(r+s){\rm  d}r.
\end{align}
Then  a similar argument as in the proof of Lemma \ref{main theorem: step2} one obtain 
\begin{align}
\|\int_{t'-s-\delta_0}^{t'-s}e^{-(t'-s-r)\A(t')}(\A(t')-\A(r+s))h(r+s){\rm  d}r\|_V
&\label{crucial step: eq3}\leq c\int_{t'-s-\delta_0}^{t'-s} \frac{\omega(t'-s-r)}{(t'-s-r)^{1+\frac{\gamma}{2}}}\d r\\&=\nonumber c\|h\|_\infty\int_{0}^{\delta_0} \frac{\omega(r)}{r^{1+\frac{\gamma}{2}}}\d r,
\end{align}
and
\begin{align}\label{crucial step: eq4}\|\int_{t-s-\delta_0}^{t-s}e^{-(t-s-r)\A(t)}(\A(t)-\A(r+s))h(r+s){\rm  d}r\|_V\leq  c\|h\|_\infty\int_{0}^{\delta_0} \frac{\omega(r)}{r^{1+\frac{\gamma}{2}}}\d r. 
\end{align}
It remains to treat integral terms in (\ref{crucial step: eq1})
	\begin{align} \nonumber I_{3,\delta_0}&:=\int_{0}^{t'-s-\delta_0} e^{-(t'-r-s)\A(t')}(\A(t')-\A(r+s))h(r+s){\rm  d}r
	\\&\nonumber\qquad\qquad\qquad-\int_{0}^{t-s-\delta_0} e^{-(t-r-s)\A(t)}(\A(t)-\A(r+s))h(r+s){\rm  d}r
	\\&=\label{proof main thm: estimation 3 integral-problem}\int_{0}^{t-s-\delta_0} \Big[e^{-(t'-r-s)\A(t')}-e^{-(t-r-s)\A(t')}\Big](\A(t')-\A(r+s))
	h(r+s){\rm  d}r
		\\&\qquad + \label{proof main thm: estimation 4 integral-problem}\int_{0}^{t-s-\delta_0} \Big[e^{-(t-r-s)\A(t')}-e^{-(t-r-s)\A(t)}\Big](\A(t')-\A(r+s))
		h(r+s){\rm  d}r
		\\&\label{proof main thm: estimation 5 integral-problem}\qquad\qquad+\int_{0}^{t-s-\delta_0} e^{-(t-r-s)\A(t)}(\A(t')-\A(t))h(r+s){\rm  d}r
		\\&\label{proof main thm: estimation 6 integral-problem}\qquad\qquad\qquad\qquad+\int_{t-s-\delta_0}^{t'-s-\delta_0} e^{-(t'-r-s)\A(t')}(\A(t')-\A(r+s))h(r+s){\rm  d}r.
	\end{align}
We denote the integral terms in  (\ref{proof main thm: estimation 3 integral-problem}), (\ref{proof main thm: estimation 4 integral-problem}), (\ref{proof main thm: estimation 5 integral-problem}) and  (\ref{proof main thm: estimation 6 integral-problem}) by $I_{3,\delta_0,1}, I_{3,\delta_0,2}, I_{3,\delta_0,3}$ and $I_{3,\delta_0,4}$ respectively.
\par\noindent Then, Lemma \ref{lemma: equivalent norms on V} and the  Proposition \ref{lemma: estimations for general from}.(\ref{analytic estimation gamma=1}) imply that 

\begin{align*}
\|\A(t')e^{-r\A(t')}z\|_V&\leq r\|A^{\frac{1}{2}}(t')\A(t')e^{-{r}\A(t')}z\|_H
= \sigma\|\A(t')e^{-\frac{r}{2}\A(t')}A^{\frac{1}{2}}(t')e^{-\frac{r}{2}\A(t')}z\|_H
\\&\leq c\frac{1}{r}\|A^{\frac{1}{2}}(t')e^{-\frac{r}{2}\A(t')}z\|_H
\\&\leq c\frac{1}{r}\|e^{-\frac{r}{2}\A(t')}z\|_V\leq \frac{c}{r^2}\|z\|_{V'}
\end{align*}
for each $z\in V'$ and $r>0.$ Therefore,
\begin{align}
\nonumber\|I_{3,\delta_0,1}\|_V&=\|\int_{0}^{t-s-\delta_0} \int_{t-r-s}^{t'-r-s}\|\A(t')e^{-\sigma\A(t')}(\A(t')-\A(t))h(r+s)\d\sigma
\d r \|_V
\\&\nonumber\leq 2M \|h\|_\infty\int_{0}^{t-s-\delta_0} \int_{t-r-s}^{t'-r-s}\|\A(t')e^{-\sigma\A(t')}\|_{\L(V',V)}\d\sigma
\d r 
\\\nonumber&\leq c \|h\|_\infty \int_{0}^{t-s-\delta_0} \int_{t-r-s}^{t'-r-s}\frac{1}{\sigma^2}\d\sigma
\d r 
\\&\leq\label{crucial step: eq6} c \|h\|_\infty \Big[\log\Big(\frac{t'-t}{\delta_0}+1\Big)+\log\Big(\frac{t-s}{t'-s}\Big) \Big]
\end{align}
\par\noindent Next, by (\ref{eq 3:Dini-condition}) and  Proposition \ref{lemma: estimations for general from}-(\ref{Eq3: estimation  semigroup}) and since $\omega: [0,T]\ra [0,\infty)$ is  bounded 
 \begin{align}
 \nonumber\|I_{3,\delta_0,2}\|_V&\leq \int_{0}^{t-s-\delta_0}\|\int_{0}^{t-s-r}e^{-((t-s-r)-\xi)\A(t)}\Big(\A(t)-\A(t')\Big)e^{-\xi\A(t')}\big(\A(t')-\A(s+r)\big)h(s+r)\d\xi\|_V\d r
 \\\nonumber&\leq c\|h\|_\infty\omega(t'-t)\int_{0}^{t-s-\delta_0}\int_{0}^{t-s-r}\frac{\omega(t-s-r)}{(t-s-r-\xi)^{\frac{1+\gamma}{2}}\xi^{\frac{1+\gamma}{2}}}\d\xi\d r
 \\\nonumber&\leq c\|h\|_\infty\omega(t'-t)\int_{0}^{t-s-\delta_0}(t-s-r)^{{-\gamma}}\Big(\int_{0}^{1}\frac{1}{(1-\xi)^{\frac{1+\gamma}{2}}\xi^{\frac{1+\gamma}{2}}}\d\xi\Big)\d r
 \\&\leq \label{crucial step: eq5} c\omega(t'-t)\|h\|_{\infty}.
 \end{align}
 
 \medskip
\noindent Again using (\ref{eq 3:Dini-condition}) and  Proposition \ref{lemma: estimations for general from}-(\ref{Eq3: estimation  semigroup}) one can estimates the last two terms $I_{3,\delta_0,3}$ and $I_{3,\delta_0,4}$ as follows
\begin{align}\nonumber\|I_{3,\delta_0,3}\|_V&\leq c\omega(t'-t)\|h\|_\infty\int_{0}^{t-s-\delta_0}(t-r-s)^{-\frac{1+\gamma}{2}}\d r
\\&\nonumber=c\omega(t'-t)\|h\|_\infty\Big[(t-s)^{\frac{1-\gamma}{2}}-\delta_0^{\frac{1-\gamma}{2}}\Big]
\\&\leq c\omega(t'-t)\|h\|_\infty\Big[T^{\frac{1-\gamma}{2}}\Big]
\leq \label{crucial step: eq7} c\omega(t'-t)\|h\|_\infty.
\end{align}
and
\begin{align}\nonumber\|I_{3,\delta_0,4}\|_V&\leq c\|h\|_\infty\int_{t-s-\delta_0}^{t'-s-\delta_0}\frac{\omega(t'-r-s)}{(t'-r-s)^{\frac{1+\gamma}{2}}}\d r 
\\&\nonumber\leq c\|h\|_\infty\sup_{\tau\in[0,T]}\frac{\omega(\tau)}{\tau^{\gamma/2}} \int_{t-s-\delta_0}^{t'-s-\delta_0}\frac{1}{(t'-r-s)^{\frac{1}{2}}}\d r 
\\&\leq \label{crucial step: eq8}c\|h\|_\infty\Big[(t'-t+\delta_0)^{1/2}-\delta_0^{1/2}\Big]. 
\end{align}
In the last inequality  we have used that $\displaystyle\sup_{\tau\in[0,T]}\frac{\omega(\tau)}{\tau^{\gamma/2}}<\infty.$ Thus we conclude that  
\begin{equation}
\|I_{3,\delta_0}\|_V\leq c\Big[\omega(t'-t)+(t'-t+\delta_0)^{1/2}-\delta_0^{1/2}+\log\Big(\frac{t'-t}{\delta_0}+1\Big)+\log\Big(\frac{t-s}{t'-s}\Big) \Big].
\end{equation} 
This estimates together with (\ref{crucial step: eq3}) and (\ref{crucial step: eq4}) proof the desired inequality.
\end{proof}

\begin{proposition}\label{main theorem1} Let $\left\{U(t,s) \big| \ (t,s)\in \Delta\right\}\subset\L(V)$ be defined by (\ref{evolution family}). Then for each fixed $s_0\in [0,T)$ the function  \[t\mapsto 
	U(t,s_0)\] is norm continuous on $(s_0,T]$ into $\L(V).$
\end{proposition}

\begin{proof} Let $s_0\in[0,T).$
 Lemma \ref{main theorem: step1} implies that 
 $I-P_{s_0}$ is  invertible on $\L(C(s_0,T,V))$ by the Neumann series. Therefore, $R_n\to 0$ as $ n\to\infty$ in $\L(C(s_0,T,V))$  where 
 \[R_n=\sum_{k=n+1}^{\infty}P_{s_0}^k,\ n\in\N\]
and, thanks to representation formula (\ref{Formula for the solution}),
\begin{align}\label{Neumann Series for the solution}
U(\cdot,s_0)x=(I-P_{s_0})^{-1}	U_1(\cdot,s_0)x=\sum_{k=0}^\infty P^k_{s_0}U_1(\cdot,s_0)x
\end{align}
 Next, proceeding by induction, we see that (\ref{main thm : main estimates}) in Lemma \ref{main theorem: step3} holds if we replace $P_{s_0}$ by $P^k_{s_0}$ for each $k\in \N\setminus\{0\}$ since  $\|P_{s_0}\|_{\L(C(s_0,T,V))}<1/4.$ The case $k=0$ is treated in Lemma \ref{main theorem: step1}.
 \par\noindent Finally, since
\begin{equation}\label{uniform estimate of U1}
\|U_1(t,s_0)x\|_V=\|e^{-(t-s_0)\A(t)}x\|_V\leq c\|x\|_V,\end{equation}
and  $\kappa_{s_0}(t-t')$ in (\ref{main thm : main estimates}) converges to $0$ as $t\to t'$ and taking into account Remark \ref{key rematk}, the proof follows by a 3-$\varepsilon$-argument.
\end{proof}
%
In the following proposition, we will proves that the mapping $U(\cdot,\cdot)$ is also norm continuous on the second variable.
\begin{proposition}\label{main theorem2}  Let $\left\{U(t,s) \big| \ (t,s)\in \Delta\right\}\subset\L(V)$ be defined by (\ref{evolution family}). Then for each fixed $t_0\in (0,T]$ the function  \[s\mapsto 
	U(t_0,s)\] is norm continuous on $[0,t_0)$ into $\L(V).$
\end{proposition}
\begin{proof}  The proof is an easy consequence of the representation formula 
(\ref{Neumann Series for the solution}).
Indeed, let $0\leq s_0\leq s<t\leq T$ and $x\in V.$ Then (\ref{eq 1:Dini-condition}), (\ref{eq 3:Dini-condition}) and  Proposition \ref{lemma: estimations for general from}-(\ref{Eq3: estimation  semigroup}) imply 
	\begin{align}
	\|(P_sU_1(\cdot,s)x-P_{s_0}U_1(\cdot,s_0)x)(t)\|_V&\leq\int_{s_0}^s \|e^{-(t-r)\A(t)}(\A(t)-\A(r))U_1(r,s_0)x\|_V{\rm  d}r\\&\leq c \Big[(t-s)^{1/2}-(t-s_0)^{1/2}\Big]\|x\|_V.
	\end{align}
On the other hand, similarly as in the first part of the proof of Lemma \ref{main theorem: step1} we obtain 
	\begin{equation}\label{main theorem2: step 1}
	\|U_1(t,s)-U_1(t,s_0)\|_{\L(V)}\leq c\log\big(\frac{t-s}{t-s_0}\big).
	\end{equation}
Finally (\ref{Neumann Series for the solution}) implies
\begin{align*}
U(\cdot,s)-U(\cdot,s_0)&=(I-P_s)^{-1}\big[U_1(\cdot,s)-U_1(\cdot,s_0)\Big]+\big[(I-P_s)^{-1}(P_s-P_{s_0})(I-P_{s_0})^{-1}U_1(\cdot,s_0)
\end{align*}
from which the claims follows.
\end{proof}

 Combining Proposition \ref{main theorem1} and  Proposition \ref{main theorem2}, we conclude that $\{U(t,s)\mid (t,s)\in \Delta\}$ is norm continuous with value in $\L(V).$  In Section \ref{Sec3 Gipps property} below we will see  that $\U$ is also norm continuous with value in $\L(H)$ provided that $V$ is compactly embedded in $H$. 
\begin{theorem}\label{main result}
Assume that  $\fra$ satisfies  (\ref{eq:continuity-nonaut})-(\ref{eq 3:Dini-condition})  and $D(A(0)^{1/2})=V.$ Let $\{U(t,s):\ (t,s)\in\Delta\}$ given by (\ref{restriction of EVF on V}). Then the function 
\[(t,s)\mapsto U(t,s)\] is a norm continuous on $\Delta$ into $\L(V).$ 
\end{theorem}
\begin{proof} Due to the evolution law $U(t,s)=U(t,r)U(r,s)$ it suffices to proof that $\{U(t,s):\ (t,s)\in\Delta\}$ is bounded on $\L(V),$ the norm continuity follows then from Proposition \ref{main theorem1} and Proposition \ref{main theorem2}. In the proof of Proposition \ref{main theorem1} we have seen that $U(t,s)$ is given by the Neumann series (\ref{Neumann Series for the solution}). This representation together with (\ref{uniform estimate of U1}) and the fact that $\|P_s\|_{\L(C(s,T,V))}\leq 1/4$ for all $0\leq s<T$ imply the claim.
	
\end{proof}
\section{Compact and Gibbs evolution family}\label{Sec3 Gipps property}
Throughout this section we adopt the notations and assumptions of Section \ref{Sec2 Norm continuity}. In this section we provide a further regularity properties of the evolution family $\{\U(t,s)\mid (t,s)\in\Delta\}.$ To this end we will introduce some definitions. 
\begin{definition}
The family $\{\U(t,s)\mid (t,s)\in\Delta\}$ is said to be a \textit{compact evolution family} on a Hilbert space $X\subseteq H$ if the operator $\U(t,s)\in \mathcal K(X)$ for every $(t,s)\in \Delta$ where $\mathcal K(X)$ denotes the space of compact operators on $X.$
\end{definition}
For strongly continuous semigroups, it is well known that compactness is a sufficient condition for norm continuity \cite[Lemma 4.22]{ENNA}. It is easy to verify that this is also true for strongly continuous evolution families. In the following we show that $\U(t,s)\in\mathcal K(H)$ and $U(t,s)\in\mathcal K(V)$ for every $(t,s)\in \Delta$ 
 whenever $V$ is compactly embedded in $H$. This would implies, in particular, that $(t,s)\mapsto \U(t,s)$ is norm continuous on $\Delta$ with value in $\L(H).$ 
 \medskip

 The following Lemma is essential for the results of this section.
 \begin{lemma}\label{key Lemma} Under the assumptions of Theorem \ref{main result}, we have $\U(t,s)\in \L(H,V)$ for every $(t,s)\in \Delta.$ 
 \end{lemma}
\begin{proof} We first proof that $\U(t,s)$ maps $H$ into $V$ for every $(t,s)\in \Delta.$ Since $\U(\cdot)$ provides the solution of (\ref{evolution equation u(s)=x}), there exists a null set $N\subset [0,T]$ such that  $\U(t,s)x\in V$ for every $s\in[0,T]$ and each $t\in]s,T]\setminus N.$ Let $x\in H$ and $(t,s)\in\Delta$ be arbitrary and fixed. Choose $t_0\in]s,T]\setminus N$ with $t_0<t.$ Then we obtain that
	\begin{equation}\label{ev equ proof}
	\U(t,s)x=U(t,t_0)\U(t_0,s)x\end{equation}
belongs to $V$ since $\U(t_0,s)\in V$ and $U(\cdot,\cdot)V\subset V.$ The boundedness of $\U$ from $H$ to $V$ follows by Banach Steinhaus Theorem using (\ref{ev equ proof}) and the strong continuity of $U(\cdot,s): [s,T]\lra V.$ 
\end{proof}
\begin{theorem}\label{compactness} Assume that  $\fra$ satisfies  (\ref{eq:continuity-nonaut})-(\ref{eq 3:Dini-condition})  and $D(A(0)^{1/2})=V.$ If $V$ is compactly embedded in $H$ then the evolution family $\{\U(t,s) \mid (t,s)\in \Delta\}$ is compact on $H,$  respectively on $V.$ In particular, $\{\U(t,s) \mid (t,s)\in \Delta\}$ is then norm continuous on $H.$
\end{theorem}
\begin{proof} The first assertion follows from the evolution family law $(ii)$ and Lemma \ref{key Lemma}. The last one follows from the remarks above. 
\end{proof}
Finally, we introduce the concept of \textit{Gipps evolution family}. For a separable Hilbert space $X$ and $p\in[1, \infty)$ the \textit{$p$-Schatten class operators} space $\S_p(X)$ is given by 
\begin{equation}\label{pschatten clas}
\mathcal S_p(X):=\Big\{ T\in \mathcal K(X) \mid\  \|T\|_{\S{_p}}:=\|(s_n)_{n\in\N}\|_{\ell^p(\N)}<\infty\Big\},
\end{equation}
where $(s_n)_{n\in \N}$ is the sequence of singular values of $T,$ that is the sequence of eigenvalue of $|T|:=(TT^*)^{1/2}.$ This classes of operators have been introduced by Robert Schatten and John
von Neumann \cite{Neu+Sch} and are also known in the literature as \textit{von Neumann-Schatten classes}. The function $\|\cdot\|_{\S_p}$ is a norm on $\S_p(X)$ called $p$-Schatten norm and, $(\S_p, \|\cdot\|_p)$ is a Banach space. For $p=1$ we obtain the will known trace class and the Hilbert-Schmidt operators for $p=2$. The $\S_p(X)$ is a $\star$-ideal in $\L(X), $ i.e., if $T\in \S_p(X)$ and $\S\in L(X)$ then $T^*\in \S_p(X),$ $TS\in \S_p(X)$ and  $ST\in \S_p(X).$ Moreover, 
\[\S_1(X)\subset \S_p(X)\subseteq \S_q(X)\subset \mathcal K(X)\]
 for every $1\leq p\leq q<\infty$ since $p\mapsto \|T\|_{\S_p}$ is non-increasing. Further, applying H\"older inequality to $\|\cdot\|_{\S_p},$ one obtain that for $T\in \S_p(X)$ and $S\in \S(X)_q,$ we have $TS\in \S_r(X)$ for $r^{-1}=p^{-1}+q^{-1}.$ For all this and more details on $p$-Schatten class we refer e.g. to  \cite{Sim05,Za2}. 
 
 The following definition generalizes the concept of Gibbs semigroups introduced by Dietrich A. Uhlenbrock \cite{Uh71} (see also \cite{Za2}) to the non-autonomous situation. 
 \begin{definition} Let $X$ be separable Hilbert space. An evolution family $\{\U(t,s) \mid (t,s)\in \Delta\}$  is said to be a Gibbs evolution family on $X$ if each $(t,s)\in \Delta$ the operator $\U(t,s)$ is of trace class.
\end{definition} 
 In view of the properties of $\S_p(X)$ and the evolution family law, any evolution family $\U(t,s)$ for which there exists $p\in(1,\infty)$ such that $\U(t,s)\in \S_p(X)$ for every $(t,s)\in \Delta$ is a Gibbs evolution family.
 \begin{theorem}\label{theorem gibbs evolution family} Assume that  $\fra$ satisfies  (\ref{eq:continuity-nonaut})-(\ref{eq 3:Dini-condition})  and $D(A(0)^{1/2})=V.$ If the embedding $V\subset H$ is of $p$-Schatten class for some $p\in (1,\infty)$ then  $\{\U(t,s) \mid (t,s)\in \Delta\}$ is a Gibbs evolution family on $H$ and $V,$ respectively. 
 \end{theorem}
 \begin{proof}
 From Lemma \ref{key Lemma} we have that $\U(t,s)\in \L(H,V)$ for each $(t,s)\in \Delta.$ The assertion follows then thanks to the ideal property of $\S_p(X)$
 \end{proof}
\section{The Laplacian with times dependent Robin boundary conditions}\label{S application}
Let $\Omega$ be a bounded domain of $\R^N$ with Lipschitz boundary $\Gamma.$ Denote by $\sigma$ the $(d-1)$-dimensional Hausdorff measure on $\Gamma.$  Let $T>0$ and $\alpha>1/4.$ Let 
\[\beta:[0,T]\times \Gamma \lra \R\]
be a bounded measurable function such that 
\[|\beta(t,x)-\beta(t,x)|\leq c|t-s|^\alpha\]
for some constant $c>0$ and every $t,s\in [0,T], x\in \Gamma.$ We consider the from $\fra:[0,T]\times V\times V\lra \C$ defined by 
\[\fra(t;u,v):=\int_\Omega\nabla u\cdot\nabla v {\rm d}x+\int_{\Gamma}\beta(t,\cdot)\restr{u}{\Gamma} \restr{v}{\Gamma} {\rm d}\sigma \]
where $u\to \restr{u}{\Gamma}: H^1(\Omega) \lra L^2(\Gamma,\sigma)$ is the trace operator. Fix $r_0\in(0,1/2)$ such that $r_0+1/2<2\alpha.$ Then the form $\fra$ satisfies (\ref{eq 3:Dini-condition}) with $\gamma:=r_0+1/2$ and $\omega(t)=t^\alpha$ since the trace operators is bounded from $H^{\gamma}(\Omega)$ with value in $H^{r_0}(\Gamma).$ The operator $A(t)$ associated with $\fra(t;\cdot,\cdot)$ on $H:=L^2(\Omega)$ is minus the Laplacian with time dependent Robin boundary conditions 
\[\partial_\nu u(t)+\beta(t,\cdot)u=0\ \text{ on } \Gamma. \]
Here $\partial_\nu$ is the weak normal derivative: let $v\in H^1(\Omega)$ such that $\Laplace u\in L^2(\Omega)$ then for $h\in L^2(\Gamma,\sigma)$ we have $\partial_\nu u:=h$ if and only if $\int_{\Omega} \nabla v\cdot \nabla w  {\rm d}x+\int_{\Omega} \Laplace v\cdot w  {\rm d}x=\int_{\Gamma} h w {\rm d}\sigma$ for all $w\in H^1(\Omega).$ Thus the domain of $A(t)$ is the set 
\[D(A(t))=\Big\{ u\in H^1(\Omega) \mid \Laplace u\in L^2(\Omega), \partial_\nu u(t)+\beta(t,\cdot)\restr{u}{\Gamma}=0 \Big\} \]
and for $u\in D(A(t)), A(t)u:=-\Laplace u.$ By Theorem \ref{compactness} the non-autonomous Cauchy problem 

 $$
 \left\{
 \begin{aligned}
 \dot {u}(t) - \Laplace u(t)& = f(t), \ u(0)\in H^1(\Omega)
 \\  \partial_\nu u(t)+\beta(t,\cdot){u}&=0 \ \text{ on } \Gamma
 \end{aligned} \right.
 $$
has $L^2$-maximal regularity in $H$ and is governed by a compact and thus norm continuous evolution family $U_R(\cdot,\cdot)$ since the embedding of $H^1(\Omega)\subset L^2(\Omega)$ is compact. Moreover, this embedding is also of $p$-Schatten class for all $p> N.$ We deduce then from Theorem \ref{theorem gibbs evolution family} that  $U_R(\cdot,\cdot)$  is in addition a Gibbs evolution family.

\end{document}